\documentclass{amsart}

\usepackage{amsfonts}
\usepackage{amssymb}
\usepackage{amsmath}

\usepackage{tikz}
\usetikzlibrary{arrows,shapes,positioning}

\newcommand{\orbit}{\text{Orb}}

\newcommand{\equivalencesbd}{\thicksim}
\newcommand{\equivalencepobd}{\thickapprox}
\newcommand{\bratleq}{\preccurlyeq}
\newcommand{\jumpofr}{\Delta_{\mathbb{R}}^+}
\newcommand{\tc}{\cong_{tc}}
\newcommand{\ntc}{\ncong_{tc}}
\newcommand{\ptc}{\cong_{tc}^*}
\newcommand{\nptc}{\ncong_{tc}^*}
\newcommand{\cantorspace}{2^{\mathbb{N}}}
\newcommand{\powersetz}{\mathcal{P}(\mathbb{Z})}
\newcommand{\natnump}{\mathbb{N}^+}
\newcommand{\cantorminimalspace}{\mathcal{M}_{\cantorspace}}
\newcommand{\A}{\mathcal{A}}

\newcommand{\Freq}{\Dens}
\newcommand{\Dens}{\text{Dens}}
\newcommand{\St}{S}
\newcommand{\setshift}{\xi}
\newcommand{\seqshift}{\sigma}

\newtheorem{theorem}{Theorem}[section]

\newtheorem{lemma}[theorem]{Lemma}
\newtheorem{corollary}[theorem]{Corollary}
\newtheorem{proposition}{Proposition}
\newtheorem{fact}[theorem]{Fact}
\newtheorem{question}{Open question}

\begin{document}
\title{The complexity of topological conjugacy of pointed Cantor minimal systems}
\author{Burak Kaya}
\address{Department of Mathematics, Rutgers University\\
110 Frelinghuysen Road, Hill Center,
Piscataway, NJ 08854, USA\\
}
\email{bkaya@scarletmail.rutgers.edu}
\keywords{Borel complexity, topological conjugacy, Cantor minimal systems, Bratteli diagrams}
\subjclass[2010]{Primary 03E15, Secondary 37B05}
\begin{abstract} In this paper, we analyze the complexity of topological conjugacy of pointed Cantor minimal systems from the point of view of descriptive set theory. We prove that the topological conjugacy relation on pointed Cantor minimal systems is Borel bireducible with the Borel equivalence relation $\jumpofr$ on $\mathbb{R}^{\mathbb{N}}$ defined by $x \jumpofr y \Leftrightarrow \{x_i:i \in \mathbb{N}\}=\{y_i:i \in \mathbb{N}\}$. Moreover, we show that $\jumpofr$ is a lower bound for the Borel complexity of topological conjugacy of Cantor minimal systems. Finally, we interpret our results in terms of properly ordered Bratteli diagrams and discuss some applications.
\end{abstract}
\maketitle

\section{Introduction}

Over the last two decades, the study of the relative complexity of classification problems has been a major focus in descriptive set theory. Under appropriate coding, the study of many classification problems can be seen as the study of the corresponding definable equivalence relations on Polish spaces. A framework has been developed and applied to many classification problems from various areas of mathematics over the years. For a detailed development of this framework, we refer the reader to \cite{Gao09}.

Topological dynamics has been one of the subjects of this study. More specifically, the topological conjugacy relations on various restricted classes of subshifts have been extensively studied. Recall that a \textit{subshift} is a closed shift invariant subset of the compact space of bi-infinite sequences over a finite alphabet together with the left shift map. Clemens \cite{Clemens09} showed that the topological conjugacy relation on subshifts over a finite alphabet is a universal countable Borel equivalence relation. Gao, Jackson, and Seward \cite{GaoJacksonSeward15} partially analyzed the topological conjugacy relation on minimal subshifts and proved that this relation is not \textit{smooth}, i.e. it is strictly more complex than the equality relation $\Delta_{\mathbb{R}}$. Thomas \cite{Thomas13} presented an elementary proof of this fact by showing that topological conjugacy of Toeplitz subshifts is not smooth. More recently, Sabok and Tsankov \cite{SabokTsankov15} have obtained results on the Borel complexity of topological conjugacy of certain subclasses of Toeplitz subshifts.

In this paper, we extend this study from minimal subshifts to arbitrary Cantor minimal systems and analyze the topological conjugacy relation on Cantor minimal systems. As far as the author knows, the Borel complexity of this relation has not been previously studied in this generality. We provide the following lower bound.

\begin{theorem}\label{theorem-mainresultunpointedcantor} $\jumpofr$ is Borel reducible to the topological conjugacy relation on Cantor minimal systems.
\end{theorem}

Theorem \ref{theorem-mainresultunpointedcantor} will be obtained as a byproduct of our analysis of topological conjugacy of \textit{pointed} Cantor minimal systems, which is the main focus of this paper. Using Stone duality, we shall show that the set of countable atomless Boolean subalgebras of $\powersetz$ which are closed under the map $A \mapsto A-1$ and whose non-empty elements are syndetic sets is a complete set of invariants for topological conjugacy of pointed Cantor minimal systems. This will enable us to prove the main result of this paper.

\begin{theorem}\label{theorem-mainresultpointedcantor} $\jumpofr$ is Borel bireducible with the topological conjugacy relation on pointed Cantor minimal systems.
\end{theorem}

It is well-known that pointed Cantor minimal systems can be represented by properly ordered Bratteli diagrams. Consequently, it follows from Theorem \ref{theorem-mainresultunpointedcantor} that equivalence of properly ordered Bratteli diagrams is Borel bireducible with $\jumpofr$. As an application of this result, we prove that there exists no Borel way of attaching orders to simple Bratteli diagrams and obtaining properly ordered Bratteli diagrams, which is invariant under equivalence of Bratteli diagrams. More precisely, we prove the following.

\begin{theorem}\label{intromaintheorem2} Let $\mathcal{SBD}$ and $\mathcal{POBD}$ be the standard Borel spaces of simple Bratteli diagrams and properly ordered Bratteli diagrams respectively; and let $\sim$ and $\approx$ denote equivalence of unordered Bratteli diagrams and ordered Bratteli diagrams respectively. Then there exists no Borel map $f: \mathcal{SBD} \rightarrow \mathcal{POBD}$ such that for all $B,B' \in \mathcal{SBD}$,
\begin{itemize}
\item[-] $f(B) \equivalencesbd B$ as unordered Bratteli diagrams and
\item[-] $f(B) \equivalencepobd f(B')$ whenever $B \equivalencesbd B'$.
\end{itemize} 
\end{theorem}

This paper is organized as follows. In Section 2, we first recall some basic facts about minimal dynamical systems and give a brief overview of the correspondence between pointed Cantor minimal systems and properly ordered Bratteli diagrams. We then recall some results from the theory of Borel equivalence relations and briefly describe the construction of the standard Borel spaces of Cantor minimal systems and properly ordered Bratteli diagrams. In Section 3, we show that pointed Cantor minimal systems can be represented by certain Boolean subalgebras of $\powersetz$ and characterize pointed minimal subshifts in terms of the generating sets of their associated Boolean algebras. In Section 4, we prove the main results of this paper. In Section 5, using results of Hjorth, Ellis, and Effros, Handelman and Shen, we show that equivalence of simple Bratteli diagrams is strictly more complex than equivalence of properly ordered Bratteli diagrams, which immediately implies Theorem \ref{intromaintheorem2}. In Section 6, we interpret our results in terms of properly ordered Bratteli diagrams and discuss some open questions and further research directions.

\section{Preliminaries}
\subsection{Topological dynamical systems and minimality}

A \textit{topological dynamical system} is a pair $(X,\varphi)$ where $X$ is a compact metrizable topological space and $\varphi: X \rightarrow X$ is a homeomorphism. Two topological dynamical systems $(X,\varphi)$ and $(Y,\psi)$ are said to be \textit{topologically conjugate} if there exists a homeomorphism $\pi: X \rightarrow Y$ such that
\[\pi \circ \varphi=\psi \circ \pi\]
Similarly, we define the class of \textit{pointed topological dynamical systems} as the class of triples of the form $(X,\varphi,x)$ where $(X, \varphi)$ is a topological dynamical system and $x \in X$. Two pointed topological dynamical systems $(X, \varphi, x)$ and $(Y, \psi, y)$ are said to be \textit{(pointed) topologically conjugate} if there exists a topological conjugacy $\pi: X \rightarrow Y$ between $(X, \varphi)$ and $(Y, \psi)$ such that $\pi(x)=y$.

Given a topological dynamical system $(X,\varphi)$, a subset $Y \subseteq X$ is said to be $\varphi$-\textit{invariant} if $\varphi[Y] = Y$. The system $(X,\varphi)$ is said to be \textit{minimal} if $(X,\varphi)$ has no non-empty proper closed $\varphi$-invariant subsets. Given a point $x \in X$ and a subset $U \subseteq X$, the set of \textit{return times} of $x$ to the subset $U$ is the set
\[ Ret_U(X,\varphi,x):=\{i \in \mathbb{Z}: \varphi^i(x) \in U\}\]
The point $x \in X$ is said to be an \textit{almost periodic point} of $(X,\varphi)$ if for every open neighborhood $U$ of $x$, the set of return times $Ret_U(X,\varphi,x)$ is \textit{syndetic}, i.e. there exists an integer $k \geq 1$ such that
\[Ret_U(X,\varphi,x)\ \cap\  \{i,i+1,\dots,i+k\}\neq \emptyset\]
for all $i \in \mathbb{Z}$. Minimality has various equivalent characterizations in terms of almost periodic points. More precisely, we have the following theorem.
\begin{fact}\cite{Kurka03}\label{theorem-minimalityconditions} Let $(X,\varphi)$ be a topological dynamical system. Then the following are equivalent.
\begin{itemize}
\item[a.] $(X,\varphi)$ is minimal.
\item[b.] For all $x \in X$, the orbit $\orbit(x)=\{\varphi^i(x): i \in \mathbb{Z}\}$ is dense in $X$.
\item[c.] For all $x \in X$, the orbit $\orbit(x)$ is dense in $X$ and $x$ is an almost periodic point.
\item[d.] For some $x \in X$, the orbit $\orbit(x)$ is dense in $X$ and $x$ is an almost periodic point.
\end{itemize}
\end{fact}

\subsection{Cantor minimal systems and properly ordered Bratteli diagrams}

A \textit{Cantor dynamical system} is a topological dynamical system $(X,\varphi)$ where $X$ is \textit{a Cantor space}, i.e. a perfect, compact, totally disconnected and metrizable topological space. From now on, we will refer to minimal Cantor dynamical systems as \textit{Cantor minimal systems}.

It is well-known that pointed Cantor minimal systems can be represented by infinite directed multigraphs known as properly ordered Bratteli diagrams. In the rest of this subsection, following \cite{HermanPutnamSkau92} and \cite{Durand10}, we shall give a brief overview of the correspondence between properly ordered Bratteli diagrams and pointed Cantor minimal systems.

An \textit{unordered Bratteli diagram} (or simply, a \textit{Bratteli diagram}) is a pair $(V,E)$ consisting of a vertex set $V$ and an edge set $E$ which can be partitioned into non-empty finite sets $V =\bigsqcup_{n=0}^{\infty} V_n$ and $E=\bigsqcup_{n=1}^{\infty} E_n$ such that the following conditions hold:
\begin{itemize}
\item[-] $V_0=\{v_0\}$ is a singleton.
\item[-] There exist a range map $r: E \rightarrow V$ and a source map $s: E \rightarrow V$ such that $r[E_n] \subseteq V_{n}$ and $s[E_n] \subseteq V_{n-1}$ for all $n \in \natnump$. Moreover, $s^{-1}(v) \neq \emptyset$ for all $v \in V$ and $r^{-1}(v) \neq \emptyset$ for all $v \in V-V_0$.
\end{itemize}

Bratteli diagrams are often given diagrammatic representations as directed graphs consisting of the vertices $V_n$ at (horizontal) level $n$ and the edges $E_{n}$ connecting the vertices at level $n-1$ with the vertices at level $n$. For an example, see Figure 1, where the orientation is taken to be in the downward direction.

\begin{center}
\begin{tikzpicture}[node distance   = 1 cm]
 
  \tikzset{VertexStyle/.style = {shape          = circle,
                                 fill           = black}}
  \tikzset{EdgeStyle/.style   = {thin}}
  \tikzset{LabelStyle/.style =   {draw,
                                  fill           = black}}
                                  
     \node[VertexStyle](v0){};
     \node[below=0.75 cm of v0]{Figure 1};
     \node[VertexStyle,right=of v0](v1){};
     \node[VertexStyle,above=1 cm of v0](v2){};
     \node[VertexStyle,above=1 cm of v1](v3){};
     \node[VertexStyle,right=of v3](v4){};
     \node[VertexStyle,left=of v2](v5){};
     \node[VertexStyle,above=1 cm of v2](v6){};
     \node[VertexStyle,above=1 cm of v3](v7){};
     \node[left=0.75 cm of v6](Vn){$V_{n-1}$};
     \node[left=0.75 cm of v5](Vn){$V_{n}$};
     \node[left=0.75 cm of v0](Vn){$V_{n+1}$};
     \node[above=0.75 cm of v4](En){$E_n$};
     \node[below=0.75 cm of v4](En1){$E_{n+1}$};

     \draw[EdgeStyle](v1) to (v3);
     \draw[EdgeStyle](v1) to (v4);
     \draw[EdgeStyle](v5) to (v1);  
     \draw[EdgeStyle](v5) to (v0);
     \draw[EdgeStyle](v2) to (v0);
     \draw[EdgeStyle](v6) to (v3);
     \draw[EdgeStyle](v6) to (v2);
     \draw[EdgeStyle](v7) to (v5);
     \draw[EdgeStyle](v7) to (v4);
     \draw[EdgeStyle](v7) to (v2);
     \tikzset{EdgeStyle/.append style = {bend left}}
     \draw[EdgeStyle](v0) to (v5); 
     \draw[EdgeStyle](v2) to (v6);
     \draw[EdgeStyle](v7) to (v4);

  \end{tikzpicture}
\end{center}

If we fix a linear order on $V_{n}$ for each $n \in \mathbb{N}$, then the edge set $E_n$ determines a $|V_n| \times |V_{n-1}|$ incidence matrix $M_n=(m_{ij})$ defined by
\[m_{ij}=|\{e \in E_n: r(e)=u_i\ \wedge\ s(e)=w_j\}|\]
where $u_i$ is the $i$-th vertex in $V_n$ and $w_j$ is the $j$-th vertex in $V_{n-1}$. For example, if we order the vertices at each level in Figure 1 from left to right, then the corresponding incidence matrices $M_{n}$ and $M_{n+1}$ are
\[M_n=
\begin{bmatrix}
0 & 1 \\
2 & 1 \\
1 & 0 \\
0 & 2 \\
\end{bmatrix}
\text{ and }
M_{n+1}=
\begin{bmatrix}
2 & 1 & 0 & 0 \\
1 & 0 & 1 & 1 \\
\end{bmatrix}
\]

Given a Bratteli diagram $(V,E)$ and $k,l \in \mathbb{N}$ with $k<l$, define $E_{k+1} \circ ... \circ E_l$ to be the set of paths from $V_k$ to $V_l$. More specifically, $E_{k+1} \circ ... \circ E_l$ is the set
\[\{(e_{k+1},...,e_l): r(e_i)=s(e_{i+1})\ i=k+1,...,l-1\ \wedge\ e_i \in E_i\ i=k+1,...,l \}\]
The corresponding range and source maps are defined by $r(e_{k+1},...,e_l):=r(e_l)$ and $s(e_{k+1},...,e_l):=s(e_{k+1})$ respectively. Observe that the product matrix $M_l \cdot ... \cdot M_{k+1}$ is the incidence matrix of the edge set $E_{k+1} \circ ... \circ E_l$.

For any sequence $0=m_0 < m_1 < m_2 < ...$ of natural numbers, we define the \textit{telescoping} of $(V,E)$ with respect to $(m_i)_{i \in \mathbb{N}}$ to be the Bratteli diagram $(V',E')$ where $V'_{n}=V_{m_n}$, $E'_n=E_{m_{n-1}+1} \circ ... \circ E_{m_n}$ and the range and source maps are defined as above. For example, if we telescope the diagram in Figure 1 to the levels $n-1$ and $n+1$, then we get the diagram in Figure 2.
\\
\begin{center}
\begin{tikzpicture}[node distance   = 1 cm]
 
  \tikzset{VertexStyle/.style = {shape          = circle,
                                 fill           = black}}
  \tikzset{EdgeStyle/.style   = {thin}}
  \tikzset{LabelStyle/.style =   {draw,
                                  fill           = black}}
                                  
     \node[VertexStyle](v0){};
     \node[below=0.5 cm of v0]{Figure 2};
     \node[VertexStyle,right=of v0](v1){};
     \node[VertexStyle,above=1.25 cm of v0](v2){};
     \node[VertexStyle,above=1.25 cm of v1](v3){};
     \node[left=1 cm of v2](Vn){$V_{n-1}$};
     \node[left=1 cm of v0](Vn){$V_{n+1}$};
     \draw[EdgeStyle](v0) to (v2);
     \draw[EdgeStyle](v1) to (v2);
     \draw[EdgeStyle](v0) to (v3);
     \draw[EdgeStyle](v1) to (v3);
     \tikzset{EdgeStyle/.append style = {bend left}}
     \draw[EdgeStyle](v0) to (v2);
     \draw[EdgeStyle](v0) to (v3);
     \draw[EdgeStyle](v1) to (v3);
     \tikzset{EdgeStyle/.append style = {bend right}}
     \draw[EdgeStyle](v0) to (v3);
     \draw[EdgeStyle](v1) to (v3);
  \end{tikzpicture}
\end{center}

A Bratteli diagram $(V,E)$ is called \textit{simple} if there exists a telescoping $(V',E')$ of $(V,E)$ such that all the incidence matrices of $(V',E')$ have only non-zero entries, i.e. every vertex of $(V',E')$ at any level is connected to every vertex at the next level. It is easily checked that $(V,E)$ is simple if and only if for every $n \in \mathbb{N}$ there exists an integer $m>n$ such that there is a path from each vertex in $V_n$ to each vertex in $V_m$.

Two Bratteli diagrams $(V,E)$ and $(V',E')$ are said to be \textit{isomorphic} if there exist bijections $f: V \rightarrow V'$ and $g: E \rightarrow E'$ which preserve the gradings and intertwine the respective source and range maps, i.e. $s' \circ g = f \circ s$ and $r' \circ g = f \circ r$. From now on, the equivalence relation on Bratteli diagrams generated by isomorphism and telescoping will be denoted by $\equivalencesbd$.

An \textit{ordered Bratteli diagram} is a triple of the form $(V,E,\bratleq)$ where $(V,E)$ is a Bratteli diagram and $\bratleq$ is a partial order on $E$ such that for all $e,e' \in E$, the edges $e$ and $e'$ are $\bratleq$-comparable if and only if $r(e)=r(e')$. Let $B=(V,E,\bratleq)$ be an ordered Bratteli diagram. We define the \textit{Bratteli compactum} associated with $B=(V,E,\bratleq)$ to be the space of infinite paths
\[X_B=\{(e_i)_{i \in \natnump}: \forall i \in \natnump\ e_i \in E_i\ \wedge\ r(e_i)=s(e_{i+1})\}\]
endowed with the topology generated by the basic clopen sets of the form
\[[e_1,e_2,\dots,e_k]_B=\{(f_i)_{i \in \natnump} \in X_B: (\forall 1 \leq i \leq k)\ e_i=f_i\}\]
It is straightforward to verify that the metric $d_B$ on $X_B$ defined by
\[d_B((e_i)_{i \in \natnump},(f_i)_{i \in \natnump})=2^{-k}\]
where $k=min\{i: e_i \neq f_i\}$ induces the same topology. We remark that the topological space $X_B$ is determined solely by $(V,E)$ and that if $(V,E)$ is a simple Bratteli diagram and $X_B$ is infinite, then $X_B$ is homeomorphic to the Cantor space.

Given an ordered Bratteli diagram $(V,E,\bratleq)$ and $k<l$ in $\mathbb{N}$, the set of paths $E_{k+1} \circ \dots \circ E_l$ from $V_k$ to $V_l$ can be given an induced lexicographic order defined by
\[(f_{k+1},f_{k+2},\dots,f_l) \prec (e_{k+1},e_{k+2},\dots,e_l)\]
if and only if for some $i$ with $k+1 \leq i \leq l$ we have $f_i \prec e_i$ and $f_j=e_j$ for all $i < j \leq l$. One readily checks that if $(V,E,\bratleq)$ is an ordered Bratteli diagram, $(V',E')$ is a telescoping of $(V,E)$, and $\bratleq'$ is the corresponding lexicographic order, then $(V',E',\bratleq')$ is an ordered Bratteli diagram. In this case, $(V',E',\bratleq')$ is called a telescoping of $(V,E,\bratleq)$. Two ordered Bratteli diagrams are said to be isomorphic if and only if there is an isomorphism of underlying unordered Bratteli diagrams which respects the partial order structure on edges. Let $\equivalencepobd$ denote the equivalence relation on ordered Bratteli diagrams generated by telescoping and isomorphism.

Given an ordered Bratteli diagram $(V,E,\bratleq)$, let $E_{max}$ and $E_{min}$ denote the sets of maximal and minimal elements of $E$ respectively. $(V,E,\bratleq)$ is said to be \textit{properly ordered} if
\begin{itemize}
\item[-] $X_B$ is infinite.
\item[-] $(V,E)$ is a simple Bratteli diagram.
\item[-] There exists a unique path $x_{min}=(e_i)_{i \in \natnump}$ such that $e_i \in E_{min}$ for all $i \in \natnump$ and there exists a unique path $x_{max}=(f_i)_{i \in \natnump}$ such that $f_i \in E_{max}$ for all $i \in \natnump$.
\end{itemize}
In this case, $x_{min}$ and $x_{max}$ are called the minimal and maximal paths respectively. (We remark that some authors require the space $X_B$ of infinite paths to be infinite as a part of the definition of an ordered Bratteli diagram to exclude Bratteli compacta which are finite.)

For every properly ordered Bratteli diagram $B=(V,E,\bratleq)$, we can define a homeomorphism $\lambda_B: X_B \rightarrow X_B$, called the \textit{Vershik map}, as follows:
\begin{itemize}
\item[-] $\lambda_B(x_{max})=x_{min}$
\item[-] $\lambda_B(e_1,\dots,e_k,e_{k+1},\dots)=(f_1,\dots,f_k,e_{k+1},\dots)$ where $k$ is the least integer such that $e_k \notin E_{max}$, $f_k$ is the successor of $e_k$ in $E$, and $(f_1,...,f_{k-1})$ is the unique minimal path in $E_1 \circ E_2 \circ \dots \circ E_{k-1}$ with range equal to the source of $f_k$.
\end{itemize}

It is routine to check that $(X_B, \lambda_B, x_{max})$ is a pointed Cantor minimal system \cite[Section 3]{HermanPutnamSkau92}. Any such dynamical system is called a \textit{Bratteli-Vershik dynamical system}. It turns out that every pointed Cantor minimal system is topologically conjugate to a Bratteli-Vershik dynamical system.

\begin{fact}\cite{HermanPutnamSkau92} \label{theorem-putnamskaumain} For any pointed Cantor minimal system $(X, \varphi, x)$ there exists a properly ordered Bratteli diagram $B=(V,E,\bratleq)$ such that $(X, \varphi,x)$ is (pointed) topologically conjugate to $(X_B, \lambda_B, x_{max})$. Moreover, if $(X_i, \varphi_i, x_i)$ corresponds to the properly ordered Bratteli diagram $B^i=(V^i,E^i,\bratleq^i)$ for $i=0,1$, then $(X_0, \varphi_0, x_0)$ is (pointed) topologically conjugate to $(X_1, \varphi_1, x_1)$ if and only if $B^0 \equivalencepobd B^1$.
\end{fact}

Given a pointed Cantor minimal system $(X, \varphi, x)$, any properly ordered Bratteli diagram $B=(V,E,\bratleq)$ such that $(X, \varphi,x)$ is topologically conjugate to $(X_B, \lambda_B, x_{max})$ will be referred to as a \textit{Bratteli-Vershik representation} of $(X, \varphi, x)$.

\subsection{Analytic and Borel equivalence relations}

A measurable space $(X,\mathcal{B})$ is called a \textit{standard Borel space} if $\mathcal{B}$ is the Borel $\sigma$-algebra of some Polish topology on $X$. An important fact that we will frequently use is that if $A \subseteq X$ is a Borel subset of a standard Borel space $(X,\mathcal{B})$, then $(A,\mathcal{B}\upharpoonright A)$ is also a standard Borel space where
\[ \mathcal{B}\upharpoonright A=\{A \cap B: B \in \mathcal{B}\}\]

Let $(X,\mathcal{B})$ and $(Y,\mathcal{B}')$ be standard Borel spaces. A map $f: X \rightarrow Y$ is called \textit{Borel} if $f^{-1}[B] \in \mathcal{B}$ for all $B \in \mathcal{B}'$. Equivalently, $f$ is Borel if and only if its graph is a Borel subset of the product space $X \times Y$.

An equivalence relation $E \subseteq X \times X$ on a standard Borel space $X$ is called a \textit{Borel equivalence relation} (respectively, an \textit{analytic equivalence relation}) if it is a Borel subset (respectively, an analytic subset) of $X \times X$. Given two analytic equivalence relations $E$ and $F$ on standard Borel spaces $X$ and $Y$ respectively, a Borel map $f: X \rightarrow Y$ is called a \textit{Borel reduction} from $E$ to $F$ if for all $x,y \in X$,
\[ x\mathbin{E}y \Longleftrightarrow f(x)\mathbin{F}f(y)\]
We say that $E$ is \textit{Borel reducible} to $F$, written $E \leq_B F$, if there exists a Borel reduction from $E$ to $F$. Observe that if $E \leq_B F$ and $F$ is Borel, then $E$ is Borel.

Two analytic equivalence relations $E$ and $F$ are said to be \textit{Borel bireducible}, written $E \sim_B F$, if both $E \leq_B F$ and $F \leq_B E$. Clearly $\sim_B$ defines an equivalence relation on the class of analytic equivalence relations. The equivalence class $[E]_{\sim_B}$ will be referred to as the \textit{Borel complexity} of $E$. Finally, we will write $E <_B F$ if both $E \leq_B F$ and $F \nleq_B E$.

Intuitively speaking, a Borel reduction from $E$ to $F$ may be regarded as an ``explicit" computation which allows us to obtain a set of complete invariants for the classification problem associated with $E$ using a set of complete invariants for the classification problem associated with $F$. Thus, if $E$ is Borel reducible to $F$, then the classification problem associated with $E$ is at most as complex as the classification problem associated with $F$.

It turns out that there are no $\leq_B$-maximal elements in the $\leq_B$-hierarchy of Borel equivalence relations. In more detail, given a Borel equivalence relation $E$ on a standard Borel space $X$, consider the Borel equivalence relation $E^+$ on the space $X^{\mathbb{N}}$ defined by
\[ x E^+ y \Leftrightarrow \{[x_n]_E: n \in \mathbb{N}\}=\{[y_n]_E: n \in \mathbb{N}\}\]
The operation $E \mapsto E^+$ is called the \textit{Friedman-Stanley jump}. That $E <_B E^+$ for Borel equivalence relations with more than one equivalence class is a result of Friedman and Stanley \cite{FriedmanStanley89}.

Let $\Delta_X$ denote the identity relation on the standard Borel space $X$. Of particular interest in this paper will be the Borel equivalence relation $\jumpofr$. Note that it follows from the Borel isomorphism theorem \cite[Corollary 1.3.8]{Gao09} that $\Delta_X^+ \sim_B \Delta_Y^+$ for any uncountable standard Borel spaces $X$ and $Y$.

Even though there are no $\leq_B$-maximal Borel equivalence relations, if we restrict our attention to \textit{countable Borel equivalence relations}, i.e. Borel equivalence relations with countable equivalence classes, then there exists a countable Borel equivalence relation $E_{\infty}$ which is \textit{universal} in the sense that for any countable Borel equivalence relation $F$ we have that $F \leq_B E$. The universal countable Borel equivalence relation $E_{\infty}$ has numerous realizations in various areas of mathematics. For example, topological conjugacy of subshifts over a finite alphabet is a universal countable Borel equivalence relation \cite{Clemens09}.

A remarkable theorem of Feldman and Moore states that any countable Borel equivalence relation $E$ on a standard Borel space $X$ is the orbit equivalence relation of a Borel action of a countable group $G$ on $X$. It easily follows from the Feldman-Moore theorem that $E_{\infty} \leq_B \jumpofr$. On the other hand, it is well-known that $\jumpofr$ is not \textit{essentially countable}, i.e. it is not Borel reducible to any countable Borel equivalence relation. Therefore, $E_{\infty} <_B \jumpofr$. (For example, see \cite[Theorem 17.1.3 and Claim 17.2.1]{Kanovei08}.)

\subsection{The standard Borel spaces of Cantor minimal systems and properly ordered Bratteli diagrams}

In order to discuss the Borel complexity of an equivalence relation on a class of structures, we need to code these structures as elements of a Polish space. In this subsection, we will briefly describe the construction of the standard Borel spaces of Cantor minimal systems and properly ordered Bratteli diagrams.

For any Cantor minimal system $(X,\varphi)$, after choosing a clopen basis for the topology of $X$, one can find a homeomorphism from $X$ to $\cantorspace$ and construct a topologically conjugate system $(\cantorspace,\psi)$. Therefore, it is sufficient to code those Cantor minimal systems which have $\cantorspace$ as their underlying topological spaces.

Let $\mathbb{B}$ be the countable atomless Boolean algebra of clopen subsets of $\cantorspace$. It is well-known that the homeomorphisms group $H(\cantorspace)$ of $\cantorspace$ and the automorphism group $Aut(\mathbb{B})$ of $\mathbb{B}$ are isomorphic via the map $\varphi \mapsto \varphi_*^{-1}$ where the \textit{dual} map of $\varphi$ is defined by $\varphi_*(U)=\varphi^{-1}[U]$ for every $U \in \mathbb{B}$. Thus, we can identify $H(\cantorspace)$ with the subspace $Aut(\mathbb{B})$ of the Polish space $\mathbb{B}^{\mathbb{B}}$. It is easily checked that $H(\cantorspace)$ is a $G_{\delta}$ subset of $\mathbb{B}^{\mathbb{B}}$ and hence $H(\cantorspace)$ is a Polish space with the induced topology. Indeed, it is a closed subgroup of the Polish group $Sym(\mathbb{B})$. Using Fact \ref{theorem-minimalityconditions}, it is straightforward to check that the set $\cantorminimalspace$ of minimal homeomorphism of $\cantorspace$ is a Borel subset of $H(\cantorspace)$ and hence is a standard Borel space. The standard Borel space of pointed Cantor minimal systems is simply $\cantorminimalspace^* := \cantorminimalspace \times \cantorspace$.

Let $\cong_{tc}$ and $\cong_{tc}^*$ denote the topological conjugacy relations on $\cantorminimalspace$ and $\cantorminimalspace^*$ respectively. It is easily seen that both $\cong_{tc}$ and $\cong_{tc}^*$ are analytic equivalence relations since they are given by the Borel actions of $H(\cantorspace)$ on the standard Borel spaces $\cantorminimalspace$ and $\cantorminimalspace^*$ respectively by conjugation.

In order to construct the standard Borel space of Bratteli diagrams, we shall code each Bratteli diagram by an element of the Polish space $(\bold{V} \times \bold{V})^{\bold{E}}$ where $\bold{V}$ and $\bold{E}$ are fixed countably infinite sets. Given a Bratteli diagram $(V,E)$, we may assume without loss of generality that $V=\bold{V}$ and $E=\bold{E}$. We then code $(V,E)$ by the function $f \in (\bold{V} \times \bold{V})^{\bold{E}}$ defined by $f(e)=(s(e),r(e))$ for each edge $e \in \bold{E}$, where $r$ and $s$ are the corresponding range and source maps. Under this coding, the subset $\mathcal{SBD}$ of $(\bold{V} \times \bold{V})^{\bold{E}}$ consisting of elements coding simple Bratteli diagrams is Borel and hence is a standard Borel space.

To construct the standard Borel space of ordered Bratteli diagrams, we need to incorporate the partial order structure on the edges. One can identify the partial order relation on the edges with an element of $2^{\bold{E} \times \bold{E}}$ and it is not difficult to check that the set of elements in $(\bold{V} \times \bold{V})^{\bold{E}} \times 2^{\bold{E} \times \bold{E}}$ coding ordered Bratteli diagrams is Borel. Given an ordered Bratteli diagram $(\mathbf{V},\mathbf{E},\bratleq)$, for each vertex $v \in \mathbf{V}$, there exists a unique path from the root $v_0$ to $v$ each element of which is in $\mathbf{E}_{min}$. It follows that if we ``mark" the minimal edges in the diagrammatic representation of $(\mathbf{V},\mathbf{E})$ together with the vertices which they connect, then we obtain a tree $\mathbf{T}_{min}$ whose edge set is exactly $\mathbf{E}_{min}$. Since $\mathbf{T}_{min}$ is finitely branching, K\"{o}nig's lemma implies that the following are equivalent
\begin{itemize}
\item[-] There is a unique infinite branch in $\mathbf{T}_{min}$.
\item[-] For every vertex $v \in \mathbf{T}_{min}$, there exists a unique successor $v^+$ of $v$ in $\mathbf{T}_{min}$ such that there exist infinitely many $w \in \mathbf{T}_{min}$ above $v^+$.
\end{itemize}
Similarly, one can argue that having a unique maximal path can be expressed with a Borel condition that only quantifies over countable sets. It easily follows the subset $\mathcal{POBD}$ of $(\bold{V} \times \bold{V})^{\bold{E}} \times 2^{\bold{E} \times \bold{E}}$ consisting of properly ordered Bratteli diagrams is Borel and hence is a standard Borel space. Notice that given an element of $\mathcal{POBD}$, we can select its unique minimal and maximal paths in a Borel way.

Let $\equivalencesbd$ and $\equivalencepobd$ denote equivalence of simple Bratteli diagrams and properly ordered Bratteli diagrams on the standard Borel spaces $\mathcal{SBD}$ and $\mathcal{POBD}$ respectively. A straightforward but tedious computation shows that both $\equivalencesbd$ and $\equivalencepobd$ are analytic equivalence relations.

\section{Representing pointed Cantor minimal systems by countable atomless $\mathbb{Z}$-syndetic algebras}

Let $\setshift$ denote the shift map on $\powersetz$ defined by $\setshift(A):=\{a-1: a \in A\}$ for all $A \in \powersetz$. A Boolean subalgebra of $\powersetz$ is said to be a $\mathbb{Z}$-\textit{syndetic algebra} if its non-empty elements are syndetic sets and it is closed under both the shift map $\setshift$ and $\setshift^{-1}$.

In this section, we will show that pointed Cantor minimal systems can be represented by countable atomless $\mathbb{Z}$-syndetic algebras and characterize minimal subshifts over finite alphabets in terms of the generating sets of their associated Boolean algebras. We shall assume familiarity of the reader with Boolean algebras and Stone duality. We refer the reader to \cite{Koppelberg89} for a general background.

Given a pointed Cantor minimal system $(X,\varphi,x)$, let $\mathbb{B}_X$ denote the Boolean algebra of clopen subsets of $X$ and define its \textit{return times algebra} $Ret(X,\varphi,x)$ to be the collection 
\[ Ret(X,\varphi,x):=\{Ret_U(X,\varphi,x): U \in \mathbb{B}_X\} \]
It is easily seen that $Ret(X,\varphi,x)$ is a Boolean subalgebra of $\powersetz$. Moreover, by Fact \ref{theorem-minimalityconditions}, the minimality of $(X,\varphi,x)$ implies that the homomorphism $U \mapsto Ret_U(X,\varphi,x)$ is injective and that $Ret(X,\varphi,x)$ is a countable atomless $\mathbb{Z}$-syndetic algebra. From now on, any countable atomless $\mathbb{Z}$-syndetic algebra will be referred to as a \textit{return times algebra}. Our choice of terminology is justified by the following lemma.

\begin{lemma}\label{lemma-returntimestopointed} If $\A$ is a return times algebra, then there exists a pointed Cantor minimal system $(X,\varphi,x)$ such that $\A=Ret(X,\varphi,x)$.
\end{lemma}

\begin{proof} Let $\St(\A)$ denote the Stone space of $\A$ consisting of ultrafilters on $\A$ topologized by the clopen sets of the form $\{w \in \St(\A): A \in w\}$ for some $A \in \A$. It is well-known that there exists a unique countable atomless Boolean algebra up to isomorphism and hence $\A$ is isomorphic to the Boolean algebra $\mathbb{B}$ of clopen subsets of $\cantorspace$. It follows from Stone duality that $\St(\A)$ is homeomorphic to $\cantorspace$.

Let $\setshift_*: \St(\A) \rightarrow \St(\A)$ be the dual homeomorphism of the automorphism $\setshift$ of $\A$ given by $\setshift_*(w):=\setshift^{-1}[w]$ and let $x_{\A} \in \St(\A)$ be the ultrafilter
\[\{A \in \A: 0 \in A\}\]
We claim that $(\St(\A),\setshift_*,x_{\A})$ is a pointed Cantor minimal system such that $\A=Ret(\St(\A),\setshift_*,x_{\A})$. For each $A \in \A$, the set of return times of $x_{\A}$ to the clopen set $U=\{w \in \St(\A): A \in w\}$ is
\begin{align*}
Ret_{U}(\St(\A),\setshift_*,x_{\A})&=\{k \in \mathbb{Z}: \setshift_*^k(x_{\A}) \in U\}\\
&=\{k \in \mathbb{Z}: A \in \setshift_*^k(x_{\A})\}\\
&=\{k \in \mathbb{Z}: k \in A\}\\
&=A
\end{align*}
It follows that $Ret(\St(\A),\setshift_*,x_{\A})=\A$ and that $x_{\A}$ is an almost periodic point. Furthermore, the orbit of $x_{\A}$ meets every non-empty clopen set and hence is dense in $\St(\A)$. Therefore $(\St(\A),\setshift_*,x_{\A})$ is a pointed Cantor minimal system by Fact \ref{theorem-minimalityconditions}.
\end{proof}

We shall refer to $(\St(\A),\setshift_*,x_{\A})$ as the \textit{ultrafilter dynamical system} associated with the return times algebra $\A$. The following lemma shows that every pointed Cantor minimal system can be represented by the ultrafilter dynamical system associated with its return times algebra.

\begin{lemma}\label{lemma-returntimesreprensentation} Let $(X,\varphi,x)$ be a pointed Cantor minimal system and let $\A$ be its return times algebra $Ret(X,\varphi,x)$. Then $(X,\varphi,x)$ is topologically conjugate to $(\St(\A),\setshift_*,x_{\A})$.
\end{lemma}

\begin{proof} Recall that the map $\rho: \A \rightarrow \mathbb{B}_X$ given by $Ret_U(X,\varphi,x) \mapsto U$ is an isomorphism of Boolean algebras. Let $\rho_*: \St(\mathbb{B}_X) \rightarrow \St(\A)$ be its dual homeomorphism given by $\rho_*(w):=\rho^{-1}[w]$ for every $w \in \St(\mathbb{B}_X)$. By Stone's theorem, we know that the map $\theta: X \mapsto \St(\mathbb{B}_X)$ given by $w \mapsto \{U \in \mathbb{B}_X: w \in U\}$ is a homeomorphism. We claim that the homeomorphism $\rho_* \circ \theta$ is a topological conjugacy between $(X,\varphi,x)$ and $(\St(\A),\setshift_*,x_{\A})$. Obviously, $(\rho_* \circ \theta)(x)=x_{\A}$. Moreover, for all $w \in X$, we have that
\begin{align*}
((\rho_* \circ\theta)\circ \varphi)(w) &= \rho_*(\{U \in \mathbb{B}_X: \varphi(w) \in U\}) \\
&= \rho_*(\{\varphi[U] \in \mathbb{B}_X: w \in U\})\\
&=\{(\rho^{-1}\circ\varphi)[U] \in \mathbb{B}_X: w \in U\}\\
&=\{Ret_{\varphi[U]}(X,\varphi,x) \in \A: w \in U\}\\
&=\setshift_*(\{Ret_{U}(X,\varphi,x) \in \A: w \in U\})\\
&=\setshift_*(\{\rho^{-1}[U] \in \mathbb{B}_X: w \in U\})\\
&=\setshift_*(\rho_*(\{U \in \mathbb{B}_X: w \in U\}))\\
&=(\setshift_* \circ (\rho_* \circ\theta))(w)
\end{align*}
\end{proof}

Consequently, the collection of return time algebras is a set of complete invariants for topological conjugacy of pointed Cantor minimal systems.

\begin{corollary}\label{corollary-returntimesinvariant} Two pointed Cantor minimal systems $(X,\varphi,x)$ and $(Y, \psi, y)$ are topologically conjugate if and only if $Ret(X,\varphi,x)=Ret(Y,\psi,y)$.
\end{corollary}
\begin{proof} Assume that $(X,\varphi,x)$ and $(Y, \psi, y)$ are topologically conjugate via the homeomorphism $\pi: X \rightarrow Y$. Since $\pi$ induces an isomorphism between $\mathbb{B}_Y$ and $\mathbb{B}_X$, we have that
\begin{align*}
Ret(X,\varphi,x)&=\{Ret_U(X,\varphi,x): U \in \mathbb{B}_X\}\\
&=\{Ret_{\pi[U]}(Y,\psi,y): U \in \mathbb{B}_X\}\\
&=\{Ret_{V}(Y,\psi,y): V \in \mathbb{B}_Y\}=Ret(Y,\psi,y)
\end{align*}
For the converse direction, assume that $Ret(X,\varphi,x)=\A=Ret(Y,\psi,y)$. Then it follows from Lemma \ref{lemma-returntimesreprensentation} that $(X,\varphi,x)$ and $(Y,\psi,y)$ are both topologically conjugate to $(\St(\A),\setshift_*,x_{\A})$.
\end{proof}

In the rest of this paper, we will often need to regard subsets of integers of the form $Ret_U(X,\varphi,x)$ as elements of $2^{\mathbb{Z}}$. From now on, the corresponding characteristic function will be denoted by $ret_U(X,\varphi,x)$.

Recall that a \textit{subshift} over a finite alphabet $\mathfrak{a}$ is a topological dynamical system $(O,\seqshift)$ where $O$ is a closed $\seqshift$-invariant subset of $\mathfrak{a}^{\mathbb{Z}}$ and $\seqshift$ is the left-shift map defined by $(\seqshift(\alpha))(i)=\alpha(i+1)$ for all $i \in \mathbb{Z}$. For notational convenience, we shall often drop the left-shift map $\seqshift$ and refer to $O$ as a subshift. For any sequence $\alpha \in \mathfrak{a}^{\mathbb{Z}}$, we define the \textit{subshift generated by} $\alpha$ to be the closure of its orbit $\orbit(\alpha)$ in $\mathfrak{a}^{\mathbb{Z}}$.

A subshift $O \subseteq \mathfrak{a}^{\mathbb{Z}}$ is said to be \textit{minimal} if the topological dynamical system $(O,\seqshift)$ is minimal. Being a closed subspace of a Cantor space, any subshift is totally disconnected, compact, and metrizable. If it is also minimal and infinite, then it has no isolated points and hence is a Cantor space itself. Thus, infinite minimal subshifts are Cantor minimal systems. Finite minimal subshifts are obviously classified up to topological conjugacy by their cardinalities. From now on, we shall exclude these trivial cases and assume that minimal subshifts are infinite.

We shall next characterize the Cantor minimal systems that are topologically conjugate to minimal subshifts over finite alphabets in terms of the generating sets of their return times algebras. We begin by noting the following trivial but useful observation.

\begin{proposition}\label{proposition-rettimecontinuous} Let $(X,\varphi)$ be a topological dynamical system and let $U$ be a clopen subset of $X$. Then the map $r_U: X \rightarrow 2^{\mathbb{Z}}$ defined by $x \mapsto ret_U(X,\varphi,x)$ is continuous. Moreover, $r_U \circ \varphi = \seqshift \circ r_U$.
\end{proposition}
\begin{proof} Since $U$ is clopen, the characteristic function $\chi_U(x): X \rightarrow 2$ is continuous and hence $r_U(x)=(\chi_U(\varphi^n(x)))_{n \in \mathbb{Z}}$ is continuous. It follows from the definition of $r_U$ that $r_U \circ \varphi = \seqshift \circ r_U$.
\end{proof}

Fix a Cantor minimal system $(X,\varphi)$. For each non-empty subset $F \subseteq \mathbb{B}_X$ consider the map $ret_F: X \rightarrow (2^{\mathbb{Z}})^F$ given by $x \mapsto (ret_U(X,\varphi,x))_{U \in F}$. The map $ret_F$ is continuous on each component by Proposition \ref{proposition-rettimecontinuous} and hence is continuous on the space $(2^{\mathbb{Z}})^F$ endowed with the product topology. Moreover, $ret_F \circ \varphi = \lambda_F \circ ret_F$ where $\lambda_F$ is the componentwise shift map defined by $\lambda_F(w)=(\sigma(w(U)))_{U \in F}$. Consider the space $(2^F)^{\mathbb{Z}}$ endowed with the product topology where each component $2^F$ has the discrete topology. Let $\eta_F$ be the map from $(2^{\mathbb{Z}})^F$ to $(2^F)^{\mathbb{Z}}$ given by
\[ (\eta_F(w)(k))(U)=(w(U))(k)\]
for all $w \in (2^{\mathbb{Z}})^F$, $U \in F$ and $k \in \mathbb{Z}$. It is easily checked that $\seqshift \circ \eta_F = \eta_F \circ \lambda_F$ and $\eta_F$ is a bijection. Moreover, $\eta_F$ is continuous whenever $F$ is finite.

It follows that if there exists a finite $F \subseteq \mathbb{B}_X$ such that $ret_F$ is injective, then $\eta_F \circ ret_F$ is a topological conjugacy from $(X,\varphi)$ onto a minimal subshift over the alphabet $2^{F}$. In order for $ret_F$ to be injective, it is sufficient for $F$ to generate $\mathbb{B}_X$ under $\varphi$ and the Boolean operations, since $\mathbb{B}_X$ separates the points of $X$. On the other hand, for each $x \in X$, the Boolean algebras $\mathbb{B}_X$ and $Ret(X,\varphi,x)$ are isomorphic via the map $U \mapsto Ret_U(X,\varphi,x)$. Hence, $\mathbb{B}_X$ is generated by finitely many elements under $\varphi$ and the Boolean operations if and only if $Ret(X,\varphi,x)$ is generated by finitely many elements under $\setshift$ and the Boolean operations for some (equivalently, every) $x \in X$.

These observations suggest the following definition. A return times algebra $\A$ is said to be \textit{finitely generated} if there exists a finite subset $F \subseteq \A$ such that $\A$ is the Boolean algebra generated by the collection $\{\setshift^k(A): A \in F\ \wedge\ k \in \mathbb{Z}\}$. In this case, the subset $F \subseteq \A$ is called a \textit{generating set} of $\A$. We are now ready to characterize pointed minimal subshifts in terms of their return times algebras.

\begin{theorem}\label{theorem-finitegeneration} Let $(X,\varphi,x)$ be a pointed Cantor minimal system. Then $(X,\varphi,x)$ is topologically conjugate to a pointed minimal subshift over some finite alphabet if and only if $Ret(X,\varphi,x)$ is finitely generated.
\end{theorem}
\begin{proof} Assume that $(X,\varphi,x)$ is topologically conjugate to a pointed minimal subshift $(O,\seqshift,w)$ over some finite alphabet $\mathfrak{a}$. Then by Corollary \ref{corollary-returntimesinvariant},
\[Ret(X,\varphi,x)=Ret(O,\seqshift,w)\]
On the other hand, since the topology of $O$ is induced by the topology of $\mathfrak{a}^{\mathbb{Z}}$, the return times algebra $Ret(O,\seqshift,w)$ is generated by the finite generating set
\[\{Ret_{U_s}(O,\seqshift,w): s \in \mathfrak{a}\}\]
where $U_s$ is the basic clopen set $\{v \in \mathfrak{a}^{\mathbb{Z}}: v(0)=s\}$. For the converse direction, assume that $Ret(X,\varphi,x)$ is finitely generated with a finite generating set $F'$. Let $F$ be the preimage of $F'$ under the map $U \mapsto Ret_U(X,\varphi,x)$. Then it follows from the previous discussion that $\eta_F \circ ret_F$ is a topological conjugacy from $(X,\varphi,x)$ onto a pointed minimal subshift over the alphabet $2^{F}$.
\end{proof}

\section{Proofs of the main results}

In this section, we will prove Theorem \ref{theorem-mainresultunpointedcantor} and Theorem \ref{theorem-mainresultpointedcantor}. We begin by noting that one direction of Theorem \ref{theorem-mainresultunpointedcantor} easily follows from Corollary \ref{corollary-returntimesinvariant}.

\begin{lemma} $\ptc\ \leq_B \jumpofr$.
\end{lemma}
\begin{proof} Recall that $\Delta_{2^{\mathbb{Z}}}^+$ and $\Delta_{\mathbb{R}}^+$ are Borel bireducible since any two uncountable standard Borel spaces are Borel isomorphic. Thus it is sufficient to prove that $\ptc\ \leq_B \Delta_{2^{\mathbb{Z}}}^+$. Let $f: \cantorminimalspace^* \rightarrow (2^{\mathbb{Z}})^{\mathbb{N}}$ be the map given by
\[(\varphi,w) \mapsto (ret_{g(i)}(\cantorspace, \varphi, w))_{i \in \mathbb{N}}\]
where $g: \mathbb{N} \rightarrow \mathbb{B}$ is a fixed enumeration of the clopen subsets of $\cantorspace$. It is straightforward to check that $f$ is a Borel map. By Corollary \ref{corollary-returntimesinvariant}, $f$ is a Borel reduction from $\ptc$ to $\Delta_{2^{\mathbb{Z}}}^+$.
\end{proof}

To show that $\jumpofr\ \leq_B\ \ptc$, it is enough to injectively assign a return times algebra to each non-empty countable subset of $\mathbb{R}$. In order to construct these return times algebras, we will need a rich collection of syndetic subsets of $\mathbb{Z}$ and these will be obtained from a non-Cantor minimal system. Fix an irrational number $\gamma \in (0,1)$ and consider the irrational rotation $T_{\gamma}: [0,1) \rightarrow [0,1)$ defined by $x \mapsto x+\gamma\ (mod\ 1)$ where $[0,1)$ is identified with the quotient $\mathbb{R}/\mathbb{Z}$. It is well-known that the topological dynamical system $([0,1),T_{\gamma})$ is minimal \cite[Proposition 1.32]{Kurka03}.

Our collection of syndetic sets will be constructed in a manner similar to the construction of \textit{Sturmian words}. A Sturmian word is a 0-1 sequence of the form $ret_{[0,\gamma)}([0,1),T_{\gamma},x)$ for some $x \in [0,1)$. The main difference will be that we do not insist that the endpoint of the half open interval be the same as the rotation angle.

Let $I \subseteq (0,1)$ be a non-empty countable set and let $\A^I$ denote the Boolean algebra consisting of the subsets of $[0,1)$ generated by the collection
\[\mathbb{G}^I=\{T^k_{\gamma}[[0,\alpha)]: k \in \mathbb{Z}\ \wedge\ \alpha \in I\}\]
\begin{proposition}\label{proposition-countableatomlessalgebra} $\A^I$ is a countable atomless Boolean subalgebra of $\mathcal{P}([0,1))$ whose non-empty elements are finite unions of half open intervals and which is closed under both $T_{\gamma}$ and $T_{\gamma}^{-1}$.
\end{proposition}
\begin{proof} Observe that complements and intersections of finite unions of half open intervals in $[0,1)$ are also finite unions of half open intervals. Since $\mathbb{G}^I$ is a countable subcollection of $\mathcal{P}([0,1))$ consisting of finite unions of half open intervals which is closed under both $T_{\gamma}$ and $T_{\gamma}^{-1}$, the same is true of the Boolean algebra $\A^I$ generated by $\mathbb{G}^I$. To see that $\A^I$ is atomless, assume to the contrary that there exists an atom $\emptyset \empty \neq A \subsetneq [0,1)$ in $\A^I$. Recall that the $T_{\gamma}$-orbit of every point is dense by the minimality of $([0,1),T_{\gamma})$. It follows that there exists $k \in \mathbb{Z} \setminus \{0\}$ such that $A \cap T_{\gamma}^k[A] \neq \emptyset$. Note that $k \gamma$ is also irrational and hence $([0,1),T_{k \gamma})$ is also minimal. Since $A$ is an atom in $\A^I$, we have that $A \cap T_{\gamma}^k[A] = A$. But then $\overline{\{T_{\gamma}^{ki}(x): i \in \mathbb{Z}\}} \subseteq \overline{A}$ for any $x \in A$ and hence $\{T_{\gamma}^{ki}(x): i \in \mathbb{Z}\}$ is not dense in $[0,1)$ for any $x \in A$, which contradicts the minimality of $([0,1),T_{k \gamma})$.
\end{proof}

Let $\A_I$ be the image of $\A^I$ under the Boolean algebra homomorphism
\[U \mapsto Ret_U([0,1),T_{\gamma},0)\]
It follows from Proposition \ref{proposition-countableatomlessalgebra} that $\A_I$ is a countable atomless subalgebra of $\powersetz$ which is closed under both $\setshift$ and $\setshift^{-1}$. By the minimality of $([0,1),T_{\gamma})$, since each $U \in \A^I$ contains an open interval, the set $Ret_U([0,1),T_{\gamma},0)$ is a syndetic subset of $\mathbb{Z}$ for every $U \in \A^I$. Hence $\A_I$ is a return times algebra.

Recall that the \textit{asymptotic density} of a subset $A$ of $\mathbb{Z}$ is defined to be the limit
\[\displaystyle \Dens(A):=\lim_{n \rightarrow \infty} \frac{|A \cap [-n,n]|}{2n+1}\]
whenever it exists. Identifying $\powersetz$ with $2^{\mathbb{Z}}$, we can similarly define the asymptotic density of an element $\alpha \in 2^\mathbb{Z}$ to be the limit
\[\displaystyle \Freq(\alpha):= \lim_{n \rightarrow \infty} \frac{|\{k \in \mathbb{Z}: \alpha(k)=1\} \cap [-n,n]|}{2n+1}\]
whenever it exists. We will next show that the set of asymptotic densities of elements of $\A_I$ is a topological conjugacy invariant for the collection of Cantor minimal systems of the form $(\St(\A_I),\setshift_*)$. We will need the following well-known equidistribution theorem.

\begin{fact}\cite{EinsiedlerWard11}\label{theorem-equidistribution} Let $\gamma \in [0,1)$ be an irrational number. Then for any $x \in [0,1)$ the sequence $(T_{\gamma}^i(x))_{i \in \mathbb{N}}$ is equidistributed in $[0,1)$ in the sense that for any $a,b \in [0,1)$ with $0 \leq a \leq b < 1$ we have that
\[ \displaystyle \lim_{n \rightarrow \infty} \frac{|\{j: 0 \leq j < n, x_j \in [a,b]\}|}{n}=b-a \]
\end{fact}
By applying Theorem \ref{theorem-equidistribution} to the irrational rotations $T_{\gamma}$ and $T_{1 - \gamma}$, it is easily checked that
\[\Dens(Ret_{U}([0,1),T_{\gamma},0))=\mu(U)\]
for every $U \in \A^I$ where $\mu$ is the usual Lebesgue measure on $[0,1)$. Having shown that elements of $\A_I$ have well-defined asymptotic densities, we define the density set of $\A_I$ to be the collection
\begin{align*}
\Dens(\A_I)&:=\{\Dens(A): A \in \A_I\}\\
&=\{\Dens(Ret_{U}([0,1),T_{\gamma},0)): U \in \A^I\}\\
&=\{\mu(U): U \in \A^I\}
\end{align*}

In order to prove that $\Dens(\A_I)$ is an invariant of the topological conjugacy class of $(\St(\A_I),\setshift_*)$, we will need the following technical lemma.

\begin{lemma}\label{lemma-frequencydoesntchange} Let $U \subseteq [0,1)$ be a finite union of half open intervals. Then for every $\alpha$ in the subshift of $2^{\mathbb{Z}}$ generated by $ret_{U}([0,1),T_{\gamma},0)$ we have that $\Freq(\alpha)=\mu(U)$.
\end{lemma}
\begin{proof} Let $\alpha$ be in the subshift generated by $ret_{U}([0,1),T_{\gamma},0)$. It is sufficient to find some $v \in [0,1)$ such that $\alpha=ret_U([0,1),T_{\gamma},v)$ since we know that
\[\Freq(ret_{U}([0,1),T_{\gamma},v))=\mu(U)\]
by the previous discussions. As $\alpha$ is in the subshift generated by $ret_{U}([0,1),T_{\gamma},0)$, there exists a sequence $(n_k)_{k \in \mathbb{N}}$ of integers such that
\[\alpha = \displaystyle \lim_{k \rightarrow \infty} \seqshift^{n_k}(ret_{U}([0,1),T_{\gamma},0))\]
Notice that
\[ \alpha = \displaystyle \lim_{k \rightarrow \infty} \seqshift^{n_k}(ret_{U}([0,1),T_{\gamma},0)) = \displaystyle \lim_{k \rightarrow \infty} (ret_{U}([0,1),T_{\gamma},T_{\gamma}^{n_k}(0)))\]
Hence, our target point $v \in [0,1)$ \textit{should} be the limit of the sequence $T^{n_k}_{\gamma}(0)$ in $[0,1)$. However, there is no reason that this sequence should converge. Nevertheless, the sequential compactness of $[0,1)$ implies that there exists some subsequence $(n_{k_i})_{i \in \mathbb{N}}$ such that $T_{\gamma}^{n_{k_i}}(0)$ is convergent, say with the limit $v=\displaystyle \lim_{i \rightarrow \infty} T_{\gamma}^{n_{k_i}}(0)$. We would like to move the limit operation inside so that
\[\displaystyle \lim_{i \rightarrow \infty} ret_U([0,1),T_{\gamma},T^{n_{k_i}}_{\gamma}(0))=ret_U([0,1),T_{\gamma},\displaystyle \lim_{i \rightarrow \infty} T^{n_{k_i}}_{\gamma}(0))=ret_U([0,1),T_{\gamma},v)\]
If the function $ret_U([0,1),T_{\gamma},\cdot)$ were continuous, then this step would be justified. However, Proposition \ref{proposition-rettimecontinuous} may fail if $U$ is not clopen and $ret_U(X,\varphi,\cdot)$ need not be continuous in general. Even though $ret_U([0,1),T_{\gamma},v)$ is not necessarily $\alpha$, we will next prove that these sequences can differ at only finitely many indices.

Let $B_v$ be the set of indices $\{j \in \mathbb{Z}: T_{\gamma}^j(v) \in \partial U \}$ where $\partial U$ denotes the boundary of $U$. Note that $\partial U$ is finite and hence $B_v$ is also finite. Otherwise, $v$ would be a periodic point of $([0,1),T_{\gamma})$, which contradicts the minimality of $([0,1),T_{\gamma})$. We will show that
\[ret_U([0,1),T_{\gamma},v) \upharpoonright (\mathbb{Z}-B_v) = \displaystyle \lim_{i \rightarrow \infty} (ret_U([0,1),T_{\gamma},T_{\gamma}^{n_{k_i}}(0)) \upharpoonright (\mathbb{Z}-B_v))\]
where the limit is taken in the topological space $2^{\mathbb{Z}-B_v}$. For each $k \geq 1$, choose $\delta_k > 0$ such that
\[\delta_k < min\{d(T_{\gamma}^j(v),y): y \in \partial U \wedge\ -k \leq j \leq k\ \wedge j \notin B_v\}\]
where $d$ is the usual metric on $\mathbb{R}/\mathbb{Z} \cong [0,1)$. Since $T_{\gamma}$ is an isometry with respect to $d$, it follows from the choice of $\delta_k$ that for any $v'$ in the open ball $B_d(v,\delta_k)$ and for any $-k \leq j \leq k$ with $j \notin B_v$, we have that
\[T_{\gamma}^j(v) \in Int(U) \Leftrightarrow T_{\gamma}^j(v') \in Int(U)\]
where $Int(U)$ is the interior of $U$. In other words, for any $v' \in B_d(v,\delta_k)$ and for any  $-k \leq j \leq k$ with $j \notin B_v$, we have that
\[ret_U([0,1),T_{\gamma},v)(j)=ret_U([0,1),T_{\gamma}, v')(j)\]
Since $v=\displaystyle \lim_{i \rightarrow \infty} T_{\gamma}^{n_{k_i}}(0)$, we know that for any $k \geq 1$, there exists $m \geq 0$ such that for all $i \geq m$ we have $|v-T_{\gamma}^{n_{k_i}}(0)| < \delta_k$. It follows that
\begin{align*}
ret_U([0,1),T_{\gamma},v) \upharpoonright (\mathbb{Z}-B_v) &= \displaystyle \lim_{i \rightarrow \infty} (ret_U([0,1),T_{\gamma},T_{\gamma}^{n_{k_i}}(0)) \upharpoonright (\mathbb{Z}-B_v))\\
&= (\displaystyle \lim_{i \rightarrow \infty} ret_U([0,1),T_{\gamma},T_{\gamma}^{n_{k_i}}(0)) \upharpoonright (\mathbb{Z}-B_v)\\
&= \alpha \upharpoonright (\mathbb{Z}-B_v)
\end{align*}
This implies that $\alpha$ and $ret_U([0,1),T_{\gamma},v)$ have the same asymptotic density $\mu(U)$.
\end{proof}

\begin{corollary} \label{corollary-frequencyinvariant} For every non-empty countable $I,J \subseteq [0,1)$, if $(\St(\A_I), \setshift_*)$ and $(\St(\A_J), \setshift_*)$ are topologically conjugate, then $\Dens(\A_I)=\Dens(\A_J)$.
\end{corollary}
\begin{proof} Assume that $(\St(\A_I), \setshift)$ and $(\St(\A_J), \setshift)$ are topologically conjugate via the homeomorphism $\pi: \St(\A_I) \rightarrow \St(\A_J)$. Let $r \in \Dens(\A_I)$. Since
\[Ret(\St(\A_I),\setshift_*,x_{\A_I}))=\A_I\]
there exists a clopen subset $W$ of $\St(\A_I)$ such that
\[ r=\Dens(Ret_W(\St(\A_I),\setshift_*,x_{\A_I}))=\Dens(Ret_{\pi[W]}(\St(\A_J),\setshift_*,\pi(x_{\A_I})))
\]
It follows from Proposition \ref{proposition-rettimecontinuous} that the image of $\St(\A_J)$ under the map
\[w \mapsto ret_{\pi[W]}(\St(\A_J),\setshift_*,w)\]
is a subshift. This subshift is minimal since it is the factor of a minimal dynamical system. Moreover, we know that
\[ret_W(\St(\A_I),\setshift_*,x_{\A_I})=ret_U([0,1),T_{\gamma},0)\] for some $U \in \A^I$; and every sequence in the subshift generated by the sequence $ret_U([0,1),T_{\gamma},0)$ has the same asymptotic density by Lemma \ref{lemma-frequencydoesntchange}. In particular,
\[r=\Dens(Ret_{\pi[W]}(\St(\A_J),\setshift_*,\pi(x_{\A_I})))=\Dens(Ret_{\pi[W]}(\St(\A_J),\setshift_*,x_{\A_J}))\]
and hence $r \in \Dens(\A_J)$. Carrying out this argument symmetrically, we obtain that
\[\Dens(\A_I)=\Dens(\A_J)\]
\end{proof}

Recall that $\Delta_{\mathcal{I}}^+ \sim_B \jumpofr$ for any uncountable Borel subset $\mathcal{I}$ of $\mathbb{R}$. Thus it is sufficient to show that $\Delta_{\mathcal{I}}^+$ is Borel reducible to both $\tc$ and $\ptc$ for some appropriately chosen Borel subset $\mathcal{I} \subseteq (0,1)$ of size continuum.

Observe that taking unions, intersections, and complements introduce no new boundary points as we generate $\A^I$ from $\mathbb{G}^I$. Hence, the set of boundary points of elements of $\A^I$ is exactly the set of boundary points of elements of $\mathbb{G}^I$ which is contained in the $\mathbb{Q}$-span of $\{1,\gamma\} \cup I$. Thus the density set $\Dens(\A_I)$ is contained in the $\mathbb{Q}$-span of $\{1,\gamma\} \cup I$ since
\[\Dens(\A_I)=\{\mu(U): U \in \A^I\}\]

\begin{lemma} \label{lemma-qlinearlyindependent} There exist an irrational number $\gamma \in (0,1)$ and a Borel subset $\mathcal{I} \subseteq (0,1)$ of size continuum such that $\mathcal{I} \cap \{1,\gamma\} = \emptyset$ and $\mathcal{I} \cup \{1,\gamma\}$ is $\mathbb{Q}$-linearly independent.
\end{lemma}
\begin{proof}
Fix a labeling of the vertices of the full binary tree of height $\omega$ by $\mathbb{N}$. For any infinite path $\alpha \in 2^{\mathbb{N}}$, let $A_{\alpha} \subseteq \mathbb{N}$ be the set of labels of the vertices that $\alpha$ passes through. Observe that intersection of any two such sets is finite. Consequently, if we\footnote{The author learned this trick from the MathOverflow post http://mathoverflow.net/q/32780 (version: 2010-07-21) by Sir Timothy Gowers, which is posted under the username ``gowers"} let $r_{\alpha}=\sum_{i=0}^{\infty} \chi_{A_{\alpha}}(i) \cdot 2^{-(i+1)^2}$ for each $\alpha \in 2^{\mathbb{N}}$, then the set $\{r_{\alpha}: \alpha \in 2^{\mathbb{N}}\}$ is a $\mathbb{Q}$-linearly independent subset of $(0,1)$ of size continuum, where $\chi_{A_{\alpha}}$ denotes the characteristic function of $A_{\alpha}$. Let $\gamma \in \{r_{\alpha}: \alpha \in 2^{\mathbb{N}}\}$ and set $\mathcal{I}:=\{r_{\alpha}: \alpha \in 2^{\mathbb{N}}\} \backslash \{\gamma\}$. Then $\mathcal{I}$ and $\gamma$ satisfy our requirements.
\end{proof}

We are now ready to prove the main theorem of this section.

\begin{theorem} \label{theorem-mainreductionconstruction} $\jumpofr$ is Borel reducible to both $\tc$ and $\ptc$.
\end{theorem}
\begin{proof} Fix some irrational number $\gamma \in (0,1)$ and a Borel subset $\mathcal{I} \subseteq (0,1)$ as in Lemma \ref{lemma-qlinearlyindependent}. Given any $\textbf{S} \in \mathcal{I}^{\mathbb{N}}$, let $f(\textbf{S})$ and $g(\textbf{S})$ be elements of $\cantorminimalspace$ and $\cantorminimalspace^*$ which code $(\St(\A_S),\setshift_*)$ and $(\St(\A_S),\setshift_*,x_{\A_S})$ respectively, where
\[S=\{\textbf{S}_i \in \mathcal{I}: i \in \mathbb{N}\}\]
and $\A_S$ is computed using the irrational rotation by $\gamma$. We will show that $f$ and $g$ are Borel reductions from $\Delta_{\mathcal{I}}^+$ to $\tc$ and $\ptc$ respectively.

We skip the tedious details of checking that $f$ and $g$ are indeed Borel maps from $\mathcal{I}^{\mathbb{N}}$ to $\cantorminimalspace$ and $\cantorminimalspace^*$. To see that $f$ and $g$ are reductions from $\Delta_{\mathcal{I}}^+$ to $\tc$ and $\ptc$ respectively, pick $\textbf{S}, \textbf{S}' \in \mathcal{I}^{\mathbb{N}}$ such that $\textbf{S}$ is $\Delta_{\mathcal{I}}^+$-equivalent to $\textbf{S}'$. Then clearly
\[ Ret(\St(\A_S),\setshift_*,x_{\A_S})=\A_S=\A_{S'}=Ret(\St(\A_{S'}),\setshift_*,x_{\A_{S'}})\]
It follows from Corollary \ref{corollary-returntimesinvariant} that $g(\textbf{S}) \ptc g(\textbf{S}')$ and hence $f(\textbf{S}) \tc f(\textbf{S}')$. Now pick $\textbf{S}, \textbf{S}' \in \mathcal{I}^{\mathbb{N}}$ such that $\textbf{S}$ is not $\Delta_{\mathcal{I}}^+$-equivalent to $\textbf{S}'$. Recall that $\Dens(\A_S)$ and $\Dens(\A_{S'})$ are contained in the $\mathbb{Q}$-spans of $\{1,\gamma\} \cup S$ and $\{1,\gamma\} \cup S'$ respectively. Moreover, we know that $S \subseteq \Dens(\A_S)$ and $S' \subseteq \Dens(\A_{S'})$. It follows from $\mathbb{Q}$-linear independence of $\mathcal{I} \cup \{1,\gamma\}$ that $\Dens(\A_S) \neq \Dens(\A_{S'})$. By Corollary \ref{corollary-frequencyinvariant}, we have that $f(\A_S) \ntc f(\A_{S'})$ and hence $g(\A_S) \nptc g(\A_{S'})$.
\end{proof}

This proves Theorem \ref{theorem-mainresultunpointedcantor} and completes the proof of Theorem \ref{theorem-mainresultpointedcantor}.

\section{From unordered Bratteli diagrams to properly ordered Bratteli diagrams}

In this section, as an application of Theorem \ref{theorem-mainresultpointedcantor}, we will prove a non-uniformity theorem regarding assigning proper orderings to simple Bratteli diagrams.

Assume that we are given an unordered Bratteli diagram $(V,E)$ such that the incidence matrices have only positive entries at each level. Then we can easily attach a partial order $\bratleq$ to $(V,E)$ as follows so that $(V,E,\bratleq)$ is a properly ordered Bratteli diagram \cite[Section 1]{Skau91}. Fix a linear order $\leq^*$ on $E$ and a linear order $\leq_n$ on $V_n$ for each $n \in \mathbb{N}$. Given $e,e' \in E_{n+1}$ with $r(e)=r(e')$, define $e \bratleq e'$ if and only if either $s(e) <_n s(e')$ or, $s(e)=s(e')$ and $e <^* e'$. It is not difficult to see that the sources of the minimal (respectively, maximal) edges are the same at every level and hence there is a unique minimal (respectively, maximal) path.

Therefore, given a simple unordered Bratteli diagram $B$, we can explicitly attach a partial order to the edges and obtain a properly ordered Bratteli diagram $B^*$, possibly after telescoping $B$. Carrying out this procedure on the relevant standard Borel spaces, one can prove that there exists a Borel map $f: \mathcal{SBD} \rightarrow \mathcal{POBD}$ such that $f(B)\equivalencesbd B$ as unordered Bratteli diagrams for every $B \in \mathcal{SBD}$. On the other hand, this map is not ``uniform" in the sense that $B_1 \equivalencesbd B_2$ does not necessarily imply $f(B_1) \equivalencepobd f(B_2)$.

One can ask whether or not such a uniform map exists. If we do not insist that $f$ be well-behaved, then we can use the axiom of choice to choose a representative from each $\sim$-class and map each $\sim$-class to the properly ordered Bratteli diagram obtained from the corresponding representative.

We will prove that there does not exist such a uniform Borel map. We first need to understand the complexity of $\sim$-equivalence of simple Bratteli diagrams. Hjorth \cite{Hjorth02} has proved that the isomorphism relation on the standard Borel space of countable torsion-free abelian groups is not Borel. Ellis showed that this relation is Borel reducible to the isomorphism relation for simple dimension groups \cite[Proposition 6.2]{Ellis10} and it essentially follows from the work of Effros, Handelman, and Shen \cite{EffrosHandelmanShen80} that the isomorphism relation for simple dimension groups is Borel reducible to the equivalence relation $\equivalencesbd$ on the space of simple Bratteli diagrams. For a detailed discussion of the latter construction, we refer the reader to \cite[Chapter 3]{Effros81}.

On the one hand, $\equivalencesbd$ is not Borel since isomorphism of countable torsion-free abelian groups is Borel reducible to it. On the other hand, $\equivalencepobd$ is Borel since the map which takes each properly ordered Bratteli diagram to the return times algebra of the corresponding Bratteli-Vershik dynamical system is a Borel reduction from $\equivalencepobd$ to $\Delta_{2^{\mathbb{Z}}}^+$. These observations immediately imply Theorem \ref{intromaintheorem2}.

\begin{proof}[Proof of Theorem \ref{intromaintheorem2}] Assume towards a contradiction that there exists a Borel map $f$ such that for all $B, B' \in \mathcal{SBD}$,
\begin{itemize}
\item[-] $f(B) \sim B$ as unordered Bratteli diagrams, and
\item[-] $B \sim B'$ implies that $f(B) \approx f(B')$.
\end{itemize}
Then $f$ is a Borel reduction from $\equivalencesbd$ to $\equivalencepobd$. This implies that $\equivalencesbd$ is Borel, which is a contradiction.
\end{proof}

\section{Concluding Remarks}
It is not difficult to prove that $\cong_{ptc}$ and $\equivalencepobd$ are Borel bireducible. For a detailed discussion of these reductions, we refer the reader to the author's dissertation \cite[Chapter 7]{Kaya16}, which closely follows the constructions given in \cite{HermanPutnamSkau92} and \cite{Durand10}.

Having determined the Borel complexity of $\equivalencepobd$, one can ask how the Borel complexity changes when we restrict our attention to various subclasses of properly ordered Bratteli diagrams. For example, what is the Borel complexity of equivalence of finite rank properly ordered Bratteli diagrams?

A Bratteli diagram $(V,E)$ is said to be of \textit{finite rank} if there exists $n \in \mathbb{N}$ such that $|V_k| \leq n$ for all $k \in \mathbb{N}$. Downarowicz and Maass \cite{DownarowiczMaass08} proved that the Bratteli-Vershik dynamical system of a properly ordered Bratteli diagram of finite rank is topologically conjugate to either an odometer, i.e. an inverse limit of a sequence of finite periodic systems, or a minimal subshift over a finite alphabet. Since topological conjugacy of odometers is smooth \cite[Theorem 7.6]{BuescuStewart95} and topological conjugacy of minimal subshifts over finite alphabets is a countable Borel equivalence relation \cite[Lemma 9]{Clemens09}, equivalence of properly ordered Bratteli diagrams of finite rank is an essentially countable Borel equivalence relation and hence is Borel reducible to $E_{\infty}$.

Theorem \ref{theorem-finitegeneration} implies that the return times algebra of a Bratteli-Vershik dynamical system arising from a finite rank properly ordered Bratteli diagram is finitely generated, unless the system is topologically conjugate to an odometer. Since the set $\mathcal{I}$ in the proof of Theorem \ref{theorem-mainreductionconstruction} was chosen to be $\mathbb{Q}$-linearly independent, the return times algebra of the pointed Cantor minimal system $(\St(\A_S),\setshift_*,x_{\A_S})$ constructed in that proof is not finitely generated unless the corresponding countable non-empty subset $S$ of $\mathcal{I}$ is finite. Hence, the properly ordered Bratteli diagrams corresponding to the pointed Cantor minimal system $(\St(\A_S),\setshift_*,x_{\A_S})$ are of infinite rank for any countably infinite $S \subseteq \mathcal{I}$. Consequently, equivalence of properly ordered Bratteli diagrams of infinite rank is Borel bireducible with $\jumpofr$. Combining these observations with the fact that $E_{\infty} <_B \jumpofr$, we obtain that equivalence of finite rank properly ordered Bratteli diagrams is strictly less complex than equivalence of infinite rank properly ordered Bratteli diagrams.

Analyzing the construction of Bratteli-Vershik representations of pointed Toeplitz subshifts with Toeplitz points described in \cite[Theorem 8]{GjerdeJohansen00} and the Borel reduction in Thomas' proof \cite{Thomas13}, one can prove that equivalence of finite rank Bratelli diagrams is not smooth. As far as the author knows, this is currently the best known lower bound for the Borel complexity of this relation.

\begin{question} What is the Borel complexity of equivalence of properly ordered Bratteli diagrams of finite rank? More generally, what is the Borel complexity of topological conjugacy of pointed minimal subshifts over finite alphabets?
\end{question}

Even though we have provided a lower bound for the Borel complexity of the topological conjugacy relation on Cantor minimal systems, we do not know any non-trivial upper bounds. The techniques used in this paper are designed to analyze pointed topological conjugacy and it is not clear to us whether or not they can be used to find any upper bounds for unpointed topological conjugacy. Thus we pose the following question.

\begin{question} What is the Borel complexity of the topological conjugacy relation on Cantor minimal systems? In particular, is this relation even Borel?
\end{question}

\textbf{Acknowledgements.} This work is largely based on the author's PhD dissertation \cite{Kaya16} under the supervision of Simon Thomas. The author is grateful to Simon Thomas and Gregory Cherlin for their invaluable guidance and many fruitful discussions. This research was partially supported by Simon Thomas and Gregory Cherlin through the NSF grants DMS-1101597 and DMS-1362974.

\bibliography{references}{}
\bibliographystyle{amsplain}
\end{document}